\documentclass{amsart}
\usepackage[utf8]{inputenc}
\usepackage{amsmath, amssymb, bbm}
\usepackage{amsthm}
\usepackage{mathtools}
\usepackage{color}
\definecolor{darkblue}{rgb}{0,0,0.6}
\usepackage[ocgcolorlinks,colorlinks=true, citecolor=darkblue, filecolor=darkblue, linkcolor=darkblue, urlcolor=darkblue]{hyperref}
\usepackage[capitalize,noabbrev]{cleveref}

\makeatletter
\@namedef{subjclassname@2010}{%
  \textup{2010} Mathematics Subject Classification}
\makeatother
\title{Long and thin covers for flow spaces}
\author{Daniel Kasprowski and Henrik R\"uping}
\address{Rheinische Friedrich-Wilhelms-Universit\"at Bonn, Mathematisches Institut,\newline\indent Endenicher Allee 60, 53115 Bonn, Germany}
\email{kasprowski@uni-bonn.de}
\email{rueping@uni-bonn.de}
\date{\today}

\newcommand{\bbR}{\mathbbm{R}}
\newcommand{\IZ}{\mathbbm{Z}}
\newcommand{\bbZ}{\mathbbm{Z}}
\newcommand{\bbN}{\mathbbm{N}}

\newcommand{\ignore}[1]{}
\newcommand{\FJCw}{{FJCw}}

\DeclareMathOperator{\ind}{ind}

\DeclareMathOperator{\axes}{axes}
\DeclareMathOperator{\cyc}{cyc}
\DeclareMathOperator{\pr}{pr}
\newcommand{\VCyc}{{\mathcal{V}cyc}}
\newcommand{\mcV}{\mathcal{V}}
\newcommand{\mcW}{\mathcal{W}}
\newcommand{\mcF}{\mathcal{F}}
\newcommand{\mcU}{\mathcal{U}}
\newcommand{\calf}{\mathcal{F}}
\newcommand{\calv}{\mathcal{V}}
\newcommand{\calu}{\mathcal{U}}
\newcommand{\cala}{\mathcal{A}}
\newcommand{\calb}{\mathcal{B}}
\newcommand{\calc}{\mathcal{C}}
\newcommand{\Fin}{{\mathcal{F}in}}

\numberwithin{equation}{section}
\newtheorem{thm}[equation]{Theorem}
\newtheorem{theorem}[equation]{Theorem}
\newtheorem{prop}[equation]{Proposition}
\newtheorem{cor}[equation]{Corollary}
\newtheorem{lemma}[equation]{Lemma}
\theoremstyle{definition}

\newtheorem{defi}[equation]{Definition}
\newtheorem{definition}[equation]{Definition}
\newtheorem{rem}[equation]{Remark}

\newtheorem{notation}[equation]{Notation}

\newtheorem{proposition}[equation]{Proposition}
\newtheoremstyle{TheoremNum}
        {}{}              
        {\itshape}                      
        {}                              
        {\bfseries}                     
        {.}                             
        { }                             
        {\thmname{#1}\thmnote{ \bfseries #3}}
\theoremstyle{TheoremNum}
\newtheorem{claim}{Claim}
\begin{document}
\begin{abstract}
Long and thin covers of flow spaces are important ingredients in the proof of the Farrell--Jones conjecture for certain classes of groups, like hyperbolic and CAT(0)-groups. In this paper we provide an alternative construction of such covers which holds in a more general setting and simplifies some of the arguments. 
\end{abstract}
\subjclass[2010]{18F25}
\keywords{Flow spaces, Long and thin covers, Farrell--Jones conjecture}
\maketitle
\section{Introduction}

In this paper $G$ will always denote a discrete and countable group. A flow space $X$ for $G$ is a metric space $X$ together with a continuous action of $G\times \bbR$, such that the action of $G=G\times 0$ on $X$ is isometric and proper. We call the action of $\bbR$ on $X$ the flow and denote the image of $x\in X$ under $t\in\bbR$ by $\Phi_tx$. See \cref{not:basics}.\ref{not:basics3}-\ref{not:basics4} for the definition of a $\VCyc$-cover.
 
\begin{theorem}[Main~Theorem]\label{thm:main}
Let $X$  be a finite-dimensional, second-countable and locally compact flow space for a group $G$ and let $\alpha,\delta$ be positive real numbers. Then there is a $\VCyc$-cover $\calu$ of $X$ of dimension at most $7\dim(X)+7$  such that for every point $x\in X$ there is an open set $U\in \calu$ with $\Phi_{[-\alpha,\alpha]}(x) \subseteq U$ and for every $U\in\calu$ there is a point $x\in X$ with $U\subseteq B_\delta(\Phi_\bbR(x))$.
\end{theorem}

The flow defines a foliation of $X$ whose leaves are the flow lines $\Phi_\bbR(x)$ for $x\in X$. Even if $X$ is a nice topological space, the orbit space $\bbR\backslash X$ can be very wild. For example $\bbR\backslash X$ will not be Hausdorff in general. When approximating the orbit map by continuous maps $f\colon X\to V$ into a simplicial complex $V$ we can therefore not expect a whole flow line to map to a single point but only to capture arbitrary large parts of the flow. By taking maps into the nerves of the covers the main theorem produces a sequence of continuous $G$-equivariant maps $f_n\colon X\to V_n$, where $V_n$ is a simplicial complex of dimension at most $7\dim(X)+7$ whose $G$-action has virtually cyclic stabilizers. More details on the construction of the maps $f_n$ are given in \cref{sec:applications}. This gives an approximation of the orbit map $X\rightarrow \bbR \backslash X$ in the following sense. For every point $x\in X$ there is a vertex $v\in V_n$ with $\Phi_{[-n,n]}(x) \subseteq f_n^{-1}(\text{St}(v))$ and for every vertex $v\in V_n$ there exists $x\in X$ such that the preimage of the star $\text{St}(v)$ is contained in $B_\delta(\Phi_{\bbR}(x))$. Note that virtually cyclic stabilizers are the smallest stabilizers one can obtain, since the stabilizers of $V_n$ not only have to contain the finite stabilizers of $X$ but also an additional infinite cyclic subgroup coming from a possible translation along the flow.

The existence of covers as in \cref{thm:main} is a main ingredient in the proof of the Farrell--Jones conjecture for hyperbolic groups by Bartels, Reich and L\"uck \cite{bartels2008k} and CAT(0)-groups by Bartels and L\"uck \cite{bartels2012borel} and Wegner \cite{wegnercat}. Long thin cell structures, a predecessor of long and thin covers, were first constructed in \cite[Section~7]{farrell1986k}. The Farrell--Jones conjecture was first introduced in \cite[Sections~1.6~and~1.7]{farrell1993}. For this application long and thin covers do not actually have to be thin, i.e.\ the condition on the cover $\mcU$ that $U\subseteq B_\delta(\Phi_\bbR(x))$ for every $U\in\mcU$ is not needed. But they cannot be too large because of the restriction on the stabilizers. The flow space $X$ is decomposed in the part with a short $G$-period and the part without, see \cref{def:period}. So far a general construction was only given for the part without short $G$-period by Bartels, Lück and Reich in \cite{bartels2008equivariant}. Here we give an alternative construction which leads to a shorter and cleaner proof.

The cover of the part with short $G$-period was previously only constructed for special groups. We give a construction that holds for all groups and thus giving a result Arthur Bartels asked for in \cite[Remark~1.5.9]{proofs}. In \cref{sec:applications} we explain how this can be used to generalize \cite[Proposition 5.11]{flow} and obtain the following corollary. See \cref{sec:applications} for appearing notation.

\begin{cor}\label{maincor} If $X$ is a cocompact, finite-dimensional flow space for the group $G$, which admits strong contracting transfers, then $G$ is strongly transfer reducible with respect to the family $\VCyc$, in particular $G$ satisfies the Farrell--Jones conjecture with finite wreath products.
\end{cor}

In \cite{busemann} the authors use this to extend the proof of the Farrell--Jones conjecture for CAT(0)-groups to a larger class of groups. In particular, giving a unified proof for hyperbolic and CAT(0)-groups and proving the Farrell--Jones conjecture for all groups acting properly and cocompactly on a finite product of hyperbolic graphs.

The proof of \cref{thm:main} will decompose the flow space into three parts; the part without a short $G$-period in \cref{sec:long}, the nonperiodic part with short $G$-period in \cref{sec:axes} and the periodic part with short $G$-period in \cref{sec:periodic}. We will construct a cover for each of the three parts and take their union. 

The construction for the cover of the part without short $G$-period is based on an idea of Arthur Bartels and Roman Sauer. We begin by constructing a countable, locally finite cover which is long in direction of the flow but has arbitrary dimension. By cutting overlapping subsets from previous elements of the cover we obtain a disjoint collection of subsets. Enlarging them in direction of the flow will produce a collection of subsets which is a long cover except for a subspace of lower dimension. Since all intersections are obtained by enlarging in direction of the $\bbR$-action we get an estimate of the dimension independent of $X$. Proceeding by induction we will cover the part without short $G$-period in at most $\dim(X)+1$ steps. To make the argument precise we will need the notion of small inductive dimension, see \cref{sec:dim}.

For the part with short $G$-period the key idea is that passing to the quotient of this subspace by the flow does not increase the dimension. This allows us to construct covers of the quotient and pull them back.

\subsection*{Acknowledgements:} We would like to thank Arthur Bartels and Roman Sauer for explaining to us their idea to use the small inductive dimension to construct covers. Furthermore, we thank Svenja Knopf, Malte Pieper and the referee for helpful comments and suggestions. The first author was supported by the Max-Planck-Society.
\newpage
\section{Basic properties of flow spaces and notations}
We will use the following notations.
\begin{notation}~\label{not:basics}
\begin{enumerate}
\item We will denote the image of $x\in X$ under the action of $g\in G$ by $gx$ and the image of $x\in X$ under the action of $t\in\bbR$ by $\Phi_tx$. The action of $\bbR$ will also be called \emph{flow}.
\item \label{not:basics3} A \emph{family $\calf$ of subgroups} of $G$ is a collection of subgroups which is closed under conjugation and taking subgroups;
\item examples are the family $\Fin$ of all finite subgroups and the family $\VCyc$ of all virtually cyclic subgroups.
\item \label{not:basics4} An \emph{$\calf$-subset} $U$ of a $G$-space is a subset with $gU\cap U\neq \emptyset \Rightarrow gU=U$ and $G_U\coloneqq \{g\mid gU=U\}\in \calf$. An \emph{$\calf$-collection} is an equivariant collection of $\calf$-subsets. An \emph{$\calf$-cover} is an $\calf$-collection which covers the whole space.
\item For a subset $B\subseteq X$ we denote by $\mathring{B}$ the interior of $B$.
\item For $U\subseteq X$ let $\partial_X U$ denote the boundary of $U$ as a subset of $X$.
\item A  $G$-action on $X$ is called \emph{cocompact} if $G\backslash X$ is compact. If the $G$-action on a flow space $X$ is cocompact we call $X$ a cocompact flow space.
\item A $G$-action on $X$ is called \emph{proper} if for every compact subspace $K\subseteq X$ the set $\{g\in G\mid K\cap gK\neq\emptyset\}$ is finite.
\end{enumerate}
\end{notation}
\begin{defi}
\label{def:period}
 For $x\in X$ define the \emph{period} of $x$ as $\inf\{t\mid t>0, x=\Phi_tx\}$. If that set is empty, we say that the period of $x$ is $\infty$. The  \emph{$G$-period} of $x$ is the period of $Gx\in G\backslash X$ with respect to the induced flow on the quotient. The \emph{flow line} through a point is its orbit under the flow.
 \end{defi}
 \begin{notation}
 \label{not:subsp}
 For $\gamma\in \bbR$ we will consider the following subspaces of $X$ satisfying the stated restrictions on the period and $G$-period. Here -- denotes that there is no condition on the period or $G$-period.\\
 \begin{center} \begin{tabular}{l|c|c}
      Notation & period & $G$-period \\\hline
      $X_{\leq \gamma}$&--&$[0,\gamma]$\\
      $X_{>\gamma}$&--&$(\gamma,\infty]$\\
      $X'_{\cyc}$& $[0,\infty)$&--\\
      $X'_{\axes}$&$\infty$&$[0,\infty)$\\
      $X_{\cyc,\gamma}$&$[0,\infty)$&$[0,\gamma]$\\
      $X_{\cyc,\gamma,>0}$&$(0,\infty)$&$(0,\gamma]$\\
      $X_{\axes,\gamma}$&$\infty$&$[0,\gamma]$
 \end{tabular}\end{center}
 All above subspaces are invariant under the $G\times \bbR$-action. We will denote the quotients of these subspaces under the $\bbR$-action by $Y_{\leq \gamma}, Y_{>\gamma}$ etc. The quotient spaces are again $G$-spaces.
 \end{notation}
 Later we will omit $\gamma$ from the notation if $\gamma$ is fixed. By \cref{lem:4.10} and \cref{lem:xcycclosed} we will see that $X_{\le \gamma}$ is topologically the disjoint union of $X_{cyc,\gamma}$ and $X_{\axes,\gamma}$. We will construct the covers for the two components separately and then take their union. 
 \begin{defi}
 Let $g\in G$. We call $c\in X$ an \emph{axis of $g$} if there is $t>0$ with $\Phi_tx=gc$. In this case we define $l(g,c):=t$. The space $X'_{\axes}$ consists of all points that are an axis for some element of $G$. Note that $l(g,\Phi_tc)=l(g,c)$ for all $t\in\bbR$, since $g\Phi_tc=\Phi_tgc=\Phi_t\Phi_{l(g,c)}c=\Phi_{l(g,c)}\Phi_tc$. Furthermore, $l(hgh^{-1},hc)$ equals $l(g,c)$.
 \end{defi}
The quotient of a metric space $X$ by a proper and isometric group action of a group $G$ is metrizable using
\[d(Gx,Gy)\coloneqq \inf_{g\in G}d(x,gy).\]
\begin{lemma}
\label{prop:sep}
Let $X$ be a metric space with a proper, cocompact and isometric $G$-action. Then $X$ is second-countable and locally compact.
\end{lemma}
\begin{proof}
Let $\pi\colon X\to G\backslash X$ be the projection. The family $\{\pi(B_{1/n}(x))\}_{x\in X}$ is an open cover of $G\backslash X$ for every $n\in \bbN$. Since $G\backslash X$ is compact there exists a finite subcover, i.e. a finite subset $I_n\subseteq X$ with $X=G(\bigcup_{x\in I_n}B_{1/n}(x))$. A countable basis for the topology is given by $\mcU\coloneqq \{gB_{1/n}(x)\mid g\in G,n\in\bbN, x\in I_n\}$.

Let $x\in X$.
\begin{claim} There exists $\epsilon>0$ such that $S:=\{g\in G\mid B_\epsilon(x)\cap gB_\epsilon(x)\neq\emptyset\}$ is finite.
\end{claim}
Otherwise there exists a sequence $x_n$ converging to $x$ and $g_n\in G$ with $g_n\neq g_m$ for $n\neq m$ such that $g_nx_n$ converges to $x$. \[K':=\{g_nx_n\mid n\in\bbN\}\cup\{x_n\mid n\in\bbN\}\cup\{x\}\] is compact and $g_nx_n\in K'\cap g_nK'$. Since the $G$-action is proper the set $\{g_n\mid n\in\bbN\}$ has to be finite, a contradiction to the assumption that all $g_n$ are different.

Now let $\epsilon>0$ be such that $S:=\{g\in G\mid B_\epsilon(x)\cap gB_\epsilon(x)\neq\emptyset\}$ is finite. And let $x_n\in B_{\epsilon/2}(x)$ be any sequence. Since $G\backslash X$ is compact there exists a subsequence $x_{n_k}$ converging in the quotient to some $z\in G\backslash X$. There is $y\in B_\epsilon(x)$ mapping to $z$ and there exist $s_n\in S$ with $s_{n_k}x_{n_k}$ converging to $y$. Since $S$ is finite there exists $s\in S$ and again a subsequence such that $sx_{n_k}$ converges to $y$. This implies that $x_{n_k}$ converges to $s^{-1}y$ and since $X$ is metric $\overline{B_{\epsilon/2}(x)}$ is therefore compact.
\end{proof}
\begin{rem} A group action on a locally compact space is proper, if and only if we can find for every point $x$ a small open neighborhood $U$ such that the set $ \{g\in G \mid gU\cap U \neq \emptyset\}$ is finite.  

In the situation where the group action is cocompact, proper and isometric \cref{prop:sep} implies that the above definition of cocompact is equivalent to the existence of a compact subset $K\subseteq X$ with $GK=X$.
\end{rem}

The following lemma will be useful to extend open covers of the subspaces from \cref{not:subsp} to the entire space.

\begin{lemma}\label{lem:openextension} Let $X$ be a metrizable space and $V,A$ be subsets such that $V$ is an open subset of $A$. Then the open subset $U_V\subseteq X$ given by
\[U_V\coloneqq \{x\in X\mid d(x,V)<d(x,A\setminus V)\}\]
 for some metric $d$ on $X$ has the following properties:
\begin{enumerate}
\item \label{lem:openextension1} We have  $U_V\cap A=V$.
\item Two such extensions $U_V,U_{V'}$ of $V,V'\subset A$ intersect if and only if $V,V'$ intersect.
\item If $G$ acts isometrically on $X$ and $A$ is a $G$-subspace, we have $gU_V=U_{gV}$. In particular,
\[\{g\in G\mid gV\cap V\neq \emptyset\}=\{g\in G\mid gU_V\cap U_V\neq \emptyset\}.\]
\item The boundary of $U_V$ intersects $A$ exactly in the boundary of $V$ (as a subspace of $A$), i.e.
\[(\partial_{X} U_V) \cap A=\partial_{A}(V).\]
\end{enumerate}
\end{lemma}
\begin{proof}~
\begin{enumerate}
\item A point $x$ in $A$ is either in $V$ in which case $d(x,V)<d(x,A\setminus V)$ since the right hand side is positive or it is not in $V$, in which case $d(x,V)\ge d(x,A\setminus V).$
\item If $V,V'$ intersect, also $U_V,U_{V'}$ intersect since they contain $V$ and $V'$ respectively. 

Now suppose $V\cap V'=\emptyset$. In this case $d(x,A\setminus V)\leq d(x,V')$ and $d(x, A\setminus V')\leq d(x,V)$ for all $x\in X$. Hence for $x\in U_V$ we get
\[d(x,A\setminus V')\leq d(x,V)<d(x,A\setminus V)\leq d(x,V)\]
and $x\notin U_{V'}$.

\item Since $G$ acts isometrically and $A$ is $G$-invariant we get
\begin{align*}
    x\in gU_V&\Leftrightarrow d(g^{-1}x,V)<d(g^{-1}x, A\setminus V)\\
    &\Leftrightarrow d(x,gV)<d(x,g(A\setminus V))=d(x,A\setminus gV)\\
    &\Leftrightarrow x\in U_{gV}.\qedhere
\end{align*}
\item Using \eqref{lem:openextension1} we obtain:
\begin{eqnarray*}
(\partial_X U_V )&=& \overline{U_V} \cap \overline{X\setminus U_V} \\
&\supset&  \overline{U_V\cap A} \cap \overline{(X\setminus U_V)\cap A}\\
&=&\overline{V} \cap \overline{A\setminus V}=\partial_A(V)\end{eqnarray*}
Thus $\partial_A(V)$ is contained in $(\partial_X U_V )\cap A$. Conversely, we have 
\begin{eqnarray*}
\partial_X(U_V)\cap A&\subset&\{x\in X\mid d(x,V)=d(x,A\setminus V)\}\cap A\\
&=&\{x\in A\mid d(x,V)=d(x,A\setminus V)=0\}\end{eqnarray*}
and the latter is just $\partial_A(V)$.
\end{enumerate}

\end{proof}

\section{Dimension theory}\label{sec:dim}
Let us recall the definition of the small inductive dimension.
\begin{defi}[{\cite[Definition 1.1.1]{engelking1978dimension}}]
To every regular space $X$ we assign the \emph{small inductive dimension $\ind(X)\in \bbN\cup\{-1,\infty\}$} given by the following properties:
\begin{enumerate}
\item $\ind (X)=-1$ if and only if $X = \emptyset$;
\item $\ind (X)\leq n$, where $n\in\bbN$ if for every point $x\in X$ and each neighborhood $V\subseteq X$ of the point $x$ there exists an open set $U\subseteq X$ such that $x\in U\subseteq V$ and $\ind(\partial U) \leq n-1$;
\item $\ind (X) = n$ if $\ind (X) \leq n$ and $\ind (X) > n-1$, i.e., the inequality $\ind (X)\leq n-1$ does not hold;
\item $\ind (X) = \infty$ if $\ind (X) > n$ for all $n\in\bbN\cup\{-1\}$.
\end{enumerate}
\end{defi}
The elementary fact that for $A,B\subset X$ we have
\[\partial(A\cap B),\partial(A\cup B),\partial(A\setminus B)\subset \partial A \cup \partial B\mbox{ and } \partial_A(A\cap B)\subset \partial_X B\]
and the following theorems will be used repeatedly to estimate the inductive dimension in the sequel.

\begin{thm}[{Subspace theorem \cite[Theorem 1.2.2]{engelking1978dimension}}]
\label{thm:subspace}
For every subspace $M$ of a regular space $X$ we have $\ind( M)\leq\ind (X)$.
\end{thm}
\begin{thm}[{Sum theorem \cite[Theorem 1.5.3]{engelking1978dimension}}]
\label{thm:sum}
If a second-countable metric space $X$ can be represented as the union of a sequence $F_k,k\in\bbN$ of closed subspaces such that $\ind(F_k)\leq n$, for every $k\in\bbN$, then $\ind (X)\leq n$.
\end{thm}
\begin{thm}[{Cartesian product theorem \cite[Theorem 1.5.16]{engelking1978dimension}}]
\label{thm:prod}
For every pair $X, Y$ of second-countable metric spaces of which at least one is non-empty we have
\[\ind(X\times Y)\leq \ind (X)+\ind(Y).\]
\end{thm}
\begin{thm}[{\cite[Theorem 1.7.7]{engelking1978dimension}}]
\label{thm:compare}
The inductive dimension of a second-countable metric space agrees with its covering dimension.
\end{thm}
\section{Boxes}
To construct the open sets for the long part, we need the notion of a box and some of its basic properties.
\begin{defi}[{\cite[Definition 2.3]{bartels2008equivariant}}]
\label{defi:box}
Let $X$ be a flow space. A \emph{box} $B$ is a subset $B\subseteq X$ with the following properties:
\begin{enumerate}
\item $B$ is a compact $\mcF in$-subset;
\item There exists a real number $l_B>0$, called the \emph{length} of the box $B$, with the property that for every $x\in B$ there exist real numbers $a_-(x)\leq 0\leq a_+(x)$ and $\epsilon(x)>0$ satisfying
\begin{align*}
l_B&=a_+(x)-a_-(x);\\
\Phi_\tau(x)\in B&\text{ for }\tau\in[a_-(x),a_+(x)];\\
\Phi_\tau(x)\notin B&\text{ for }\tau\in(a_-(x)-\epsilon(x),a_-(x))\cup(a_+(x),a_+(x)+\epsilon(x)).
\end{align*}
\end{enumerate}
\end{defi}
To a box $B$ we can assign the \emph{central slice} \[S_B\coloneqq \{x\in B\mid a_-(x)+a_+(x)=0\}.\]
We abuse notation and define $\partial S_B\coloneqq \partial B\cap S_B$ and $\mathring{S}_B\coloneqq  \mathring{B}\cap S_B$.
\begin{lemma}[{\cite[Lemma 2.6]{bartels2008equivariant}}]
\label{lemma:homeo}
The map
\[\mu_B\colon S_B\times [-l_B/2,l_B/2]\xrightarrow{\cong} B,\quad (x,t)\mapsto \Phi_t(x)\]
is a $G_B$-homeomorphism.
\end{lemma}
Consequently, we can define a projection $\pr_B$ to the central slice, via 
\[\pr_B:B \rightarrow S_B,\quad x\mapsto \pr_1\circ \mu_B^{-1}(x).\]
By definition of $\mu_B$ this is the same as $x\mapsto \Phi_{-\pr_2(\mu_B^{-1}(x))}(x)$.

\begin{definition} An \emph{open box} is the interior of a box. 
\end{definition}

\begin{lemma}[{\cite[Lemma 2.16]{bartels2008equivariant}}]
\label{lemma:box}
For every $x\in X_{>\gamma}$ and for every $0<l\le \gamma$ there exists a box $B$ of length $l$ with $x\in \mathring{S}_B$ and $G_B=G_x$.
\end{lemma}
\begin{rem}
In \cite[Lemma 2.16]{bartels2008equivariant} the space $X\setminus X^\bbR$ is assumed to be locally connected. This is only needed to find a box as in \cref{lemma:box} with the additional assumption that $S_B$ is connected. Furthermore, in \cite[Lemma 2.16]{bartels2008equivariant} the lemma is only stated for $l<\gamma$. Since $x\in X_{>\gamma}$ is also in $X_{>\gamma+\varepsilon}$ for $\varepsilon$ small enough, it follows that the lemma also holds for $l=\gamma$.
\end{rem}

\begin{lemma}\label{lem:locfinboxes} Let $X$ be a locally compact, second-countable flow space and $\alpha,\delta>0$. 
Then there is a countable collection of compact subsets $(S_i)_{i\in \mathbbm{N}}$ and $(x_i)_{i\in \mathbbm{N}}$ with $x_i\in X$
such that:
\begin{enumerate}
\item\label{lem:locfinboxes:a} $\Phi_{[-10\alpha,10\alpha]}(S_i)\subseteq B_\delta(\Phi_{[-10\alpha,10\alpha]}(x_i))$ is a box of length $20\alpha$ with central slice $S_i$;
\item\label{lem:locfinboxes:b} The interiors of the smaller boxes $\Phi_{[-\alpha,\alpha]}(S_i)$ form a locally finite cover of $X_{>20\alpha}$.
\item\label{lem:locfinboxes:c} For any two $S_i,S_j$ the set
\begin{equation*} \{t\in [-3\alpha,3\alpha]\mid \exists x\in S_i : \Phi_t(x)\in S_j\}\subseteq \mathbbm{R}\end{equation*}
has diameter less than $\alpha$.
\end{enumerate}
\end{lemma}
\begin{proof}
By Lemma~\ref{lemma:box} we can find for every point $x\in X_{>20\alpha}$ a box $\Phi_{[-10\alpha,10\alpha]}S'_x$ of length $20\alpha$ such that $x$ is in the interior of this box and in the central slice. Furthermore, we can choose them in such a way that $S'_{gx}=gS'_x$. We can assume
\[\Phi_{[-10\alpha,10\alpha]}(S'_x)\subseteq B_\delta(\Phi_{[-10\alpha,10\alpha]}(x)),\]
since otherwise we can replace $S_x'$ by $S'_x\cap\bigcap_{t\in [-10\alpha,10\alpha]}\Phi_{-t}(B_\delta(\Phi_tx))$.
Now consider the open cover
\[\{(\Phi_{[-\alpha,\alpha]}S'_x)^\circ\mid x\in X \}\]
and push it along the quotient map $\pi$ to the quotient $G\backslash X$. The quotient $G\backslash X$ is metrizable and hence second-countable and paracompact by \cite[Corollary~2.1.8]{pears1975dimension}. We can thus find a countable, locally finite refinement $\{V(n)\mid n\in \bbN\}$ of this cover. Being a refinement means that we can find for every $n$ an $x(n)$ with $\pi^{-1}(V(n))\subseteq \pi^{-1}(\pi((\Phi_{[-\alpha,\alpha]}S'_{x(n)})^\circ))=G\cdot (\Phi_{[-\alpha,\alpha]}S'_{x(n)})^\circ$.
Now define a cover 
\[\calv\coloneqq \{\pi^{-1}(V(n))\cap g(\Phi_{[-\alpha,\alpha]}S'_{x(n)})^\circ\mid g\in G, n\in \mathbbm{N}\}.\]
This is a countable, locally finite, $G$-invariant open cover of $X$. Choose an enumeration $\calv =\{V_i\mid i\in \mathbbm{N}\}$. We can enlarge these sets by first projecting the closure of $V_i=\pi^{-1}(V(n_i))\cap g_i(\Phi_{[-\alpha,\alpha]}S'_{x(n_i)})^\circ$ to the central slice $gS'_{x(n_i)}$ and then letting it flow by $[-\alpha,\alpha]$. Call the resulting
box $C_i$ and its central slice $D_i$. Let $\calc$ be the collection of boxes $\{C_i\mid i\in \mathbbm{N}\}$.

To show that it is locally finite at some point $x$, pick a compact neighborhood $K$ and note that if $K\cap C_i\neq  \emptyset$, then $\Phi_{[-2\alpha,2\alpha]}(K)\cap \overline{V_i}\neq \emptyset$. Since $\Phi_{[-2\alpha,2\alpha]}(K)$ is compact and the collection $\{\overline{V_i}\mid i\in \mathbbm{N}\}$ locally finite, this can happen only for finitely many $i$.

It remains to establish \cref{lem:locfinboxes}\eqref{lem:locfinboxes:c}. To achieve this we have to subdivide the central slices $D_i$ into finitely many compact sets $S_{i,1},\ldots,S_{i,m_i}$.
We will do this by induction over $i \in \bbN$. If there is a $g\in G$ such that $gD_j=D_i$ for some $j<i$, define 
$S_{i,k}\coloneqq gS_{j,k}$. Otherwise proceed as follows.

Define for $j\in \mathbbm{N}$ a continuous function
\[f_{i,j}: D_i\cap \Phi_{[-3\alpha,3\alpha]}(D_j)\rightarrow [-3\alpha,3\alpha],\quad x\mapsto t \mbox{ with }\Phi_t(x)\in D_j.\]
There is precisely one such $t$ since there is a box of length $6\alpha$ with central slice $D_j$.
This map is continuous by \cite[Definition~4.14]{bartels2008equivariant}. For every $x\in D_i\cap \Phi_{[-3\alpha,3\alpha]}(D_j)$ there is a small, open $\Fin$-neighborhood $U_{i,j,x}$ in $D_i\cap \Phi_{[-3\alpha,3\alpha]}(D_j)$ such that $f_{i,j}(U_{i,j,x})$ has diameter less than $\alpha$. This neighborhood can be extended by \cref{lem:openextension} to an open neighborhood $U'_{i,j,x}$ in $D_i$ such that $U'_{i,j,x}\cap \Phi_{[-3\alpha,3\alpha]}(D_j)=U_{i,j,x}$. 

The set $J_i=\{j\mid D_i\cap \Phi_{[-3\alpha,3\alpha]}(D_j)\neq \emptyset \}$ is finite and thus
$U_{i,x}\coloneqq \bigcap_{j\in J_i} U'_{i,j,x}$ is still an open neighborhood. Let $W_{i,x}\coloneqq \bigcap_{h \in G_{D_i}}h^{-1} U_{i,hx}$. This collection is $G_{D_i}$-invariant. Since $D_i$ is compact we can find a finite subcover $W_{i,x_1},\ldots, W_{i,x_{m_i}}$. This can be chosen in a $G_{D_i}$-equivariant way. This yields a $\Fin$-cover of $D_i$. The new collection 
\[\{S_i\}_{i\in\bbN}\coloneqq \{\overline{W_{n,x_k}}\mid n\in \mathbbm{N},1\le k\le m_i\}.\]
does the job. Note that, since every element of the collection $\{S_i\}_{i\in\bbN}$ is a subset of some $S'_x,~x\in X$ considered in the beginning, we have $\Phi_{[-10\alpha,10\alpha]}(S_i)\subseteq B_\delta(\Phi_{[-10\alpha,10\alpha]}(x))$ for some $x\in X$.
\end{proof}
\section{Covering \texorpdfstring{$X_{>\gamma}$}{the part with large G-period}}\label{sec:long}
We will now construct covers for the part without a short $G$-period. Here $X$ denotes a second-countable, locally compact flow space of dimension $n$ and let $\alpha>0$ be given. Let $\gamma$ be $20\alpha$ and $(S_i)_{i\in \mathbbm{N}}$ be a collection of compact subsets of $X$ 
as in \cref{lem:locfinboxes}. Fix these choices for the rest of this section.
\begin{lemma}\label{lem:thelemma} Let $(A_i)_{i\in \bbN}$ be a collection where $A_i$ is a compact $G_{S_i}$-invariant subset of $\mathring{S}_i$ of inductive dimension at most $k$ for some $k\ge 0$. Then there is a collection of open $G_{S_i}$-invariant subsets $B_i\subseteq \mathring{S}_i$ such that
\begin{enumerate}
\item \label{lem:thelemma1}$A_i\setminus \bigcup_{j\in \mathbbm{N},g\in G}\Phi_{(-3\alpha,3\alpha)}(gB_j)$ has inductive dimension at most $k-1$ for every $i$;
\item \label{lem:thelemma2}Any point is contained in at most $5$ sets of the collection \[\{\Phi_{(-4\alpha,4\alpha)}(gB_i)\mid i\in \mathbbm{N},g\in G\}.\]
\end{enumerate}
\end{lemma}
\begin{proof}
Let \[A'\coloneqq  \bigcup_{j\in \mathbbm{N},g\in G} \Phi_\mathbbm{R}(gA_j);\]
\[A_i'\coloneqq  S_i\cap A'= \bigcup_{j\in \mathbbm{N}, z\in \mathbbm{Z},g\in G} S_i \cap \Phi_{[(z-1)\alpha,z\alpha]}gA_j.\]
 Since $S_i$ and thus also $\Phi_{1-z}g^{-1}S_i$ is the central slice of a box of length greater than $\alpha$, the projection $\Phi_{[0,\alpha]}A_j\to A_j$ restricted to $\Phi_{1-z}g^{-1}S_i\cap \Phi_{[0,\alpha]}A_j$ is injective. The intersection $\Phi_{1-z}g^{-1}S_i\cap \Phi_{[0,\alpha]}A_j$ is compact and $A_j$ is Hausdorff. Hence $\Phi_{1-z}g^{-1}S_i\cap \Phi_{[0,\alpha]}A_j$ is homeomorphic to its image in $A_j$. Therefore, the union in the displayed equation above is a countable union of compact spaces homeomorphic to subspaces of $A_j$. Thus by \cref{thm:subspace} and \cref{thm:sum} we have that $\ind(A_i')\le k$. 

For every $x\in A_i$ we can find an open neighborhood $U_x\subseteq A_i'$ such that we have $\ind (\partial_{A_i'} U_x)\le k-1$ and $\partial_{A_i'} U_x\subseteq \mathring{S}_i$. We can choose those such that $U_{gx}=gU_x$ for $g$ in the finite group $G_{S_i}$ by replacing $U_x$ by $\bigcap_{g\in G_{S_i}} g^{-1}U_{gx}$.

By compactness we can find a finite $G_{S_i}$-subset $F_i\subseteq A_i$ such that $V_i\coloneqq \bigcup_{x\in F_i}U_x$ contains $A_i$. By \cref{lem:openextension}, the open $G_{S_i}$-subset
\[U_i\coloneqq  \{x\in S_i\mid d(x,V_i)<d(x,A_i'\setminus V_i) \} \]
of $S_i$ has the following properties:
\begin{itemize}
\item $U_i\cap A_i' =V_i;$
\item $(\partial_{S_i} U_i) \cap A_i'=\partial_{A_i'}(U_i\cap A_i')$.
\end{itemize}
Thus,
\begin{equation}\ind((\partial_{S_i} U_i)\cap A_i')=\ind(\partial_{A_i'}(V_i))\le \ind(\bigcup_{x\in F_i} \partial_{A_i'}(U_x))\le k-1 ,\label{eq:indpartUi}\end{equation} 
where the last inequality follows from \cref{thm:sum} since $\partial_{A_i'}(U_x)$ is closed in $A_i'$.

Define inductively $G_{S_i}$-invariant subsets
\begin{eqnarray*}
B_i&\coloneqq &U_i \setminus \bigcup_{\substack{j<i, g\in G\\ \Phi_{[-2\alpha,2\alpha]} g\overline{B_j}\cap U_i\neq \emptyset} }\Phi_{(-3\alpha,3\alpha)}g\overline{B_j}\\
&=&U_i \setminus (S_i\cap \bigcup_{\substack{j<i, g\in G\\ \Phi_{[-2\alpha,2\alpha]} g\overline{B_j}\cap U_i\neq \emptyset} }\Phi_{(-3\alpha,3\alpha)}g\overline{B_j}).
\end{eqnarray*}
Now since $S_i$ and $g\overline{B_j}$ are both compact and the $G$-action on the space is proper, there are only finitely many $g\in G$ such that the intersection $\Phi_{[-2\alpha,2\alpha]} g\overline{B_j}\cap U_i$ is not empty. We want to show by induction on $i$ that $\ind(A_i'\cap\partial_{S_i}B_i)\le k-1$: 
\begin{eqnarray*}
\partial_{S_i}B_i&\subseteq& \partial_{S_i} U_i\cup \bigcup_{\substack{j<i, g\in G\\ \Phi_{[-2\alpha,2\alpha]} g\overline{B_j}\cap U_i\neq \emptyset} } \partial_{S_i}(S_i\cap \Phi_{(-3\alpha,3\alpha)}g\overline{B_j})\\
&\subseteq &\partial_{S_i} U_i\cup \bigcup_{\substack{j<i, g\in G\\ \Phi_{[-2\alpha,2\alpha]} g\overline{B_j}\cap U_i\neq \emptyset} } S_i\cap \partial_{X}( \Phi_{(-3\alpha,3\alpha)}g\overline{B_j})\\
&\subseteq &\partial_{S_i} U_i\cup \bigcup_{\substack{j<i, g\in G\\ \Phi_{[-2\alpha,2\alpha]} g\overline{B_j}\cap U_i\neq \emptyset} } S_i\cap (\Phi_{\{\pm 3\alpha\}}g\overline{B_j}\cup \Phi_{(-3\alpha,3\alpha)}(g\partial_{S_j}B_j)).
\end{eqnarray*}
By \cref{lem:locfinboxes}\eqref{lem:locfinboxes:c} and $\Phi_{[-2\alpha,2\alpha]} g\overline{B_j}\cap U_i\neq \emptyset$, we get that $S_i\cap \Phi_{\{\pm (3\alpha)\}}g\overline{B_j}$ is empty. 

Consider the following equation, where the first and fourth equality are from expanding the definitions of $A_i'$ and $A'$, the second equality follows from $A'$ being $\bbR$-invariant, the third equality follows from $\partial_{S_j}B_j\subseteq S_j$ and the last equality is given by writing $\bbR$ as a union of the intervals $[z,z+1]$.
\begin{eqnarray*}&&A_i'\cap S_i\cap  \Phi_{[-3\alpha,3\alpha]}(g\partial_{S_j}B_j)\\
&=& S_i\cap A'\cap \Phi_{[-3\alpha,3\alpha]}(g\partial_{S_j}B_j)\\
&=& S_i\cap \Phi_{[-3\alpha,3\alpha]}(A'\cap g\partial_{S_j}B_j)\\
&=& S_i\cap \Phi_{[-3\alpha,3\alpha]}(A'\cap gS_j\cap g\partial_{S_j}B_j)\\
&=& \bigcup_{h\in G, k\in\bbN} S_i\cap \Phi_{[-3\alpha,3\alpha]}(\Phi_\mathbbm{R}(hA_k)\cap gS_j\cap g\partial_{S_j}B_j)\\
&=& \bigcup_{h\in G, k\in\bbN, z\in\bbZ}S_i\cap \Phi_{[-3\alpha,3\alpha]}(\Phi_{[z,z+1]}(hA_k)\cap gS_j\cap g\partial_{S_j}B_j).
\end{eqnarray*}
Arguing as at the beginning of the proof, the compact space \[S_i\cap \Phi_{[-3\alpha,3\alpha]}(\Phi_{[z,z+1]}(hA_k)\cap gS_j\cap g\partial_{S_j}B_j)\] is homeomorphic to a subspace of $(\Phi_{[z,z+1]}(hA_k)\cap gS_j)\cap g\partial_{S_j}B_j\subseteq g(A_j' \cap \partial_{S_j}B_j)$.
Thus we know that $A_i'\cap \partial_{S_i}(B_i)$ is contained in the union of
a countable collection compact subsets of $A_j' \cap \partial_{S_j}(B_j)$ with  $j<i$ which have inductive dimension at most $k-1$ by induction assumption and the space $\partial_{S_i}U_i\cap A_i'$ whose inductive dimension is at most $k-1$ by \eqref{eq:indpartUi}. We would like to apply \cref{thm:sum}, but $\partial_{S_i}U_i\cap A_i'$ might not be a closed subspace of $S_i$. This problem can be overcome by writing $A_i'$ as a union of countably many compact spaces as in the beginning of the proof. Thus we obtain 
\[\ind (A_i'\cap \partial_{S_i}(B_i))\le k-1.\]
This completes the induction.

Now let us show \eqref{lem:thelemma1}.
We have the following inclusion. The individual steps are explained afterwards.
\begin{eqnarray*}
&&A_i \setminus \bigcup_{j\in \mathbbm{N},g\in G}\Phi_{(-3\alpha,3\alpha)}(gB_j)\\
&\subseteq &A_i \setminus \bigcup_{j\le i ,g\in G}\Phi_{(-3\alpha,3\alpha)}(gB_j)\\
&=&(A_i \setminus \Phi_{(-3\alpha,3\alpha)}(U_i \setminus \hspace{-0.8cm}\bigcup_{\substack{j<i, g\in G\\ \Phi_{[-2\alpha,2\alpha]} g\overline{B_j}\cap U_i\neq \emptyset} }\hspace{-0.8cm}\Phi_{(-3\alpha,3\alpha)}g\overline{B_j}))\setminus\bigcup_{j< i ,g\in G}\Phi_{(-3\alpha,3\alpha)}(gB_j)\\
&\subseteq &(A_i \setminus (U_i \setminus \hspace{-0.8cm}\bigcup_{\substack{j<i, g\in G\\ \Phi_{[-2\alpha,2\alpha]} g\overline{B_j}\cap U_i\neq \emptyset} }\hspace{-0.8cm}\Phi_{(-3\alpha,3\alpha)}g\overline{B_j}))\setminus\bigcup_{j< i ,g\in G}\Phi_{(-3\alpha,3\alpha)}(gB_j)\\
&\subseteq&(A_i\cap (\bigcup_{j<i, g\in G}\Phi_{(-3\alpha,3\alpha)}g\overline{B_j}))\setminus\bigcup_{j< i ,g\in G}\Phi_{(-3\alpha,3\alpha)}(gB_j)\\
&\subseteq &A_i\cap \bigcup_{j<i, g\in G}\Phi_{(-3\alpha,3\alpha)}g(\overline{B_j}\setminus B_j)\\
&\subseteq &\bigcup_{j<i, g\in G}A_i\cap \Phi_{[-3\alpha,3\alpha]}g\partial_{S_j}{B_j}.\\
\end{eqnarray*}
The first and second inclusion comes from removing a smaller set. The first equality is given by removing first $\Phi_{(-3\alpha,3\alpha)}gB_i$ and inserting the definition of $B_i$. The third inclusion follows from the fact that $A_i\subseteq U_i$.  The last two inclusions are obvious.

The set $A_i\cap \Phi_{[-3\alpha,3\alpha]}g\partial_{S_j} B_j$ is homeomorphic to a compact subset of $A'\cap\partial_{S_j} gB_j=g(A'_j\cap\partial_{S_j} B_j)$ and thus its inductive dimension is at most $k-1$. So $A_i \setminus \bigcup_{j\in \mathbbm{N},g\in G}\Phi_{(-3\alpha,3\alpha)}(gB_j)$ embeds into a space of inductive dimension at most $k-1$. Thus we have shown \eqref{lem:thelemma1}.

To show \eqref{lem:thelemma2} we first want to show that the collection
\[\calc\coloneqq \{\Phi_{[-\alpha,\alpha]}gB_i\mid g\in G,i\in \mathbbm{N}\}\]
consists of pairwise disjoint sets. Since $B_j$ is a $G_{S_j}$-invariant subset of the central slice of a box of length $2\alpha$, we know that $g\Phi_{[-\alpha,\alpha]}B_j\cap g'\Phi_{[-\alpha,\alpha]}B_j\neq \emptyset$ if and only if $gg'^{-1}\in G_{S_j}$, in which case the two sets are equal. 
By definition of $B_i$ we have that for $j<i$
\[g\Phi_{[-\alpha,\alpha]}B_i\cap g'\Phi_{[-\alpha,\alpha]}B_j= \emptyset.\]
Thus the collection $\calc$ consists of pairwise disjoint sets. Now let us consider the collection
\[\calc'=\{\Phi_{[-4\alpha,4\alpha]}gB_i\mid g\in G,i\in \mathbbm{N}\}.\]
If $x$ is contained in $\Phi_{[-4\alpha,4\alpha]}gB_i$ then there is a $\beta\in \{-4\alpha,-2\alpha,0\alpha,2\alpha,4\alpha\}$ such that $\Phi_\beta(x)\in \Phi_{[-\alpha,\alpha]}gB_i$ and thus $x$ is contained in at most $5$ sets of $\calc'$.
\end{proof}

\begin{theorem}\label{thm:long} For every $\alpha,\delta>0$ there is a $\Fin$-cover of $X_{>20\alpha}$ of dimension at most $5(\ind(X)+1)$  with the following property.
For every point $x\in X_{>20\alpha}$ there is an open set in this cover containing $\Phi_{[-\alpha,\alpha]}(x)$ and for every open set $U$ in this cover there is an element $x\in X$ with $U\subseteq B_\delta(\Phi_\bbR(x))$.
\end{theorem}

\begin{proof} Let $\gamma\coloneqq 20\alpha$. First consider the collection of subsets $\mathfrak{A}_0\coloneqq \{\mathring{S}_i \mid i\in \mathbbm{N}\}$ as in \cref{lem:locfinboxes}. We can find a sequence $\varepsilon_i>0$ such that $A_i^0\coloneqq \{x\in S_i\mid d(x,\partial S_i)\ge \varepsilon_i\}$ has the property that
\[\{\Phi_{[-2\alpha,2\alpha]}(gA_i^0)\mid g\in G,i\in \mathbbm{N}\}\]
still covers the whole of $X_{>\gamma}$. Note that $A_i^0$ is a compact, $G_{S_i}$-invariant subset of $\mathring{S}_i$ of inductive dimension at most $\ind(X)$.

By Lemma~\ref{lem:thelemma} we can find a collection $B_i^0\subseteq \mathring{S}_i$ as in the lemma. Define new compact subsets 
\[A_i^1\coloneqq  A_i^0\setminus \bigcup_{j\in \mathbbm{N},g\in G}\Phi_{(-3\alpha,3\alpha)}(gB_j^0)\]
and iterate the process. Note that $A_i^k=\emptyset$ for $k>\ind(X)$ since its dimension is $-1$.
By \cref{lem:thelemma}~\eqref{lem:thelemma1} we know that any point in $X_{>\gamma}$ is contained in an open set of the form
$\Phi_{(-3\alpha,3\alpha)}(gB_i^k)$ for some $g\in G,i\in \mathbbm{N},k\in 0,\ldots,\ind(X)$.
Now consider the collection
\[\calb\coloneqq \{\Phi_{(-4\alpha,4\alpha)}B_i^k\mid i\in \mathbbm{N},k\in 0,\ldots, \ind(X)\}.\]
Thus for every point $x$ we can find an open set $U\in \calb$ with $\Phi_{[-\alpha,\alpha]}(x)\in U$.
Every point is contained in at most $5\cdot (\ind(X)+1)$ sets of this collection by \cref{lem:thelemma}~\eqref{lem:thelemma2}. By construction each element in $\calb$ is a subset of $\Phi_{[-4\alpha,4\alpha]}(S_i)$ and in particular contained in $B_\delta(\Phi_{[-10\alpha,10\alpha]}(x))$ for some $x\in X$ by \cref{lem:locfinboxes}~\eqref{lem:locfinboxes:a}.
\end{proof}

\section{Axes with bounded G-period}\label{sec:axes}
In this section we will construct the cover for $X_{\axes,\gamma}$. From now on we will fix $\gamma>0$ and omit it from the notation.

\begin{lemma}\label{lem:4.10}
The $G$-subspace $X_{\axes}\subseteq X$ is
\begin{enumerate}
\item \label{4.10i}closed,
\item \label{4.10iii}second-countable and
\item \label{4.10iv}locally compact.
\end{enumerate}
\end{lemma}
First we need to prove that \cite[Lemma 4.6, Corollary 4.7]{flow} still hold without the assumption that the metric space is proper.
\begin{lemma}
\label{lem:4.6}
Let $(Z,d)$ be a metric space with a proper isometric $G$-action. If $(z_n)_{n\in\bbN}$ and $(g_n)_{n\in\bbN}$ are sequences in $Z$ and $G$ such that $z_n$ converges to $z\in Z$ and $g_nz_n$ converges to $z'\in Z$, then $\{g_n\mid n\in\bbN\}$ is finite and for every $g\in G$ such that $g_n=g$ for infinitely many $n\in\bbN$ we have $gz=z'$.
\end{lemma}
\begin{proof}
Define $K:=\{z_n\}\cup\{g_nz_n\}\cup\{z,z'\}$. Then $K$ is compact and $g_nK\cap K\neq \emptyset$, thus the set $\{g_n\mid n\in\bbN\}$ is finite. If $g_n=g$ for infinitely many $n\in\bbN$, then $z'=\lim_{n\to\infty}g_nz_n=\lim_{n\to\infty}gz_n=gz$. 
\end{proof}
\begin{cor}
\label{cor:4.7}
Let $(Z,d)$ be a metric space with a proper isometric $G$-action. If $L\subseteq Z$ is compact, then $HL\subseteq Z$ is closed for any subset $H\subseteq G$. 
\end{cor}
\begin{proof}
Let $h_nz_n$ be converging to $z$ with $h_n\in H,z_n\in L$. After passing to a subsequence $z_n$ converges to $z'\in L$ and by \cref{lem:4.6} we can pass to a further subsequence with $h_n\equiv h$. Thus $z=hz'\in HL$.
\end{proof}
\begin{proof}[Proof of \cref{lem:4.10}]
\eqref{4.10i} Let $c_n\in X_{\axes}$ be a sequence that converges to $c\in X$. There are $g_n\in G,t_n\in(0,\gamma]$ such that $g_nc_n=\Phi_{t_n}c_n$. We can pass to a subsequence and assume that $t_n$ converges to $t$. Then $g_nc_n=\Phi_{t_n}c_n$ converges to $\Phi_{t}c$. Since $G$ acts properly and isometrically on $X$ we can apply \cref{lem:4.6} and assume after passing to a subsequence that $g_n=g$ is constant. We have \[gc=\lim g_nc_n=\lim\Phi_{t_n}c_n=\Phi_{t}c.\] Since the group action is proper and $g$ has infinite order, $t$ can not be zero. 
\eqref{4.10iii} Subspaces of second-countable spaces are again second-countable. 
\eqref{4.10iv} Closed subspaces of locally compact spaces are again locally compact.
\end{proof}

\begin{lemma}\label{lem:flowlinesclosed}If there are $g\in G,0<t\in\bbR$ with $\Phi_tc=gc$, then $\Phi_{\mathbbm{R}}c$ is closed.
\end{lemma}
\begin{proof}
We have $\Phi_\mathbbm{R}c=\bigcup_{n\in \mathbbm{Z}}g^n\Phi_{[0,t]}c$. The group $\langle g\rangle $ also acts properly and isometrically. The set $\Phi_{[0,t]}c$ is compact. Hence $\Phi_\mathbbm{R}c$ is closed by \cref{cor:4.7}.
\end{proof}

\begin{lemma}
\label{lem:proper} The subspace $\{t\in\bbR\mid \Phi_t L\cap L'\neq\emptyset\}$ of $\bbR$ is compact
for every two compact subspaces $L,L'\subseteq X_{\axes}$.
\end{lemma}
\begin{proof}
Since this set is closed, it suffices to show that it is bounded. Furthermore it is a subset of $\{t\in\bbR\mid \Phi_t (L\cup L')\cap (L\cup L')\neq\emptyset\}$ and thus it suffices to consider the case $L'=L\neq \emptyset$. In this case the set is symmetric at $0$ and thus it suffices to find an upper bound.

The set 
\[S=\{g\in G\mid  gL\cap \Phi_{[-\gamma,\gamma]}L\neq \emptyset\}\]
is finite since the group action is proper. Furthermore we have $S=S^{-1}$. 
Since every point in $L$ is an axis for some group element, the set $S$ contains at least one element of infinite order. Let $m$ be the maximal integer such that there is an element $g\in S$ of infinite order with $g^m\in S$.
Now let $t\ge 0$ be given such that there is an $x\in L$ with $\Phi_t(x)\in L$. Since $x$ is an axis, we can find a $g\in G$ with $gx=\Phi_{l(g,x)}x$ for $0<l(g,x)\le \gamma$. Hence $g$ has infinite order and $g\in S$. Now write $t$ in the form $t=m'l(g,x)+r$ with $m'\in \IZ,r\in [0,l(g,x)]$. By assumption $\Phi_t(x)\in L$. Furthermore we have
\[\Phi_t(x)=g^{m'}\Phi_r(x)\in g^{m'} \Phi_{[-\gamma,\gamma]}(L).\]
Thus $g^{-m'}\in S$ and by symmetry we have $g^{m'}\in S$. Hence $m'\le m$ and thus 
\[t=m'l(g,x)+r\le (m'+1)\gamma\le (m+1)\gamma.\qedhere\]
\end{proof}

\begin{lemma}
\label{lem:metrizable}
The space $Y_{\axes}$ is locally compact and metrizable.
\end{lemma}
\begin{proof}
The space $X_{\axes}$ is second-countable by \cref{lem:4.10}. The quotient map $p\colon X_{\axes}\to Y_{\axes}$ is open, because it is the quotient by the action of the group $\bbR$.
Let $y$ be a point in $Y_{\axes}$ and $c\in X_{\axes}$ be a preimage. Let $U$ be an open neighborhood of $y$. Let $L\subseteq p^{-1}(U)$ be a compact neighborhood of $c$ and since $p$ is continuous and open $p(L)$ is a compact neighborhood of $y$. Thus $Y_{\axes}$ is locally compact.

Points in $Y_{\axes}$ are closed, since $\Phi_\bbR c\subseteq X$ is closed for every $c\in X_{\axes}$ by \cref{lem:flowlinesclosed}. For a closed subset $A\subseteq Y_{\axes}$ and $p(c)\notin A$ there is an $\epsilon>0$ such that $\overline{B_\epsilon(c)}$ is compact and $B_\epsilon(c)\cap p^{-1}(A)=\emptyset$ and thus also $p(B_\epsilon(c))\cap A=\emptyset$. 
\begin{claim}
The set $B\coloneqq p(\overline{B_{\epsilon/2}(c)})$ is closed.
\end{claim}
Then the complement of $B$ is an open neighborhood of $A$ and it is disjoint from the open neighborhood $p(B_{\epsilon/2}(c))$ of $p(c)$. Hence, $Y_{\axes}$ is regular. $Y_{\axes}$ is second-countable since it is a quotient of a subspace of a second-countable space.
By Urysohn's metrization theorem \cite[Theorem~34.1]{munkres} the quotient $Y_{\axes}$ is metrizable.

It remains to prove the claim: Let $c_i$ be a sequence in $p^{-1}(p(\overline{B_{\epsilon/2}(c)}))$ which is converging to $c'$ in $X_{\axes}$. Let $t_n\in\bbR$ be such that $\Phi_{t_n}c_n\in \overline{B_{\epsilon/2}(c)}$. Let $\delta>0$ be such that $\overline{B_\delta(c')}$ is compact and $N\in\bbN$ be such that $d(c_n,c')<\delta$ for all $n\geq N$. By \cref{lem:proper} there is a $t_0>0$ such that $\Phi_t(\overline{B_{\delta}(c')})\cap \overline{B_{\epsilon/2}(c)}=\emptyset$ for all $|t|>|t_0|$ and thus $|t_n|\leq|t_0|$ for all $n\geq N$. Therefore, there is a subsequence $t_n$ converging to $t'\in\bbR$. Thus for $n$ large enough we have $\Phi_{t'}c_n\in\Phi_{[-\epsilon,\epsilon]}(\overline{B_{\epsilon/2}(c)})$, which is compact. So also the limit $\Phi_{t'}c'=\lim_{n\to\infty}\Phi_{t'}c_n$ lies in $\Phi_{[-\epsilon,\epsilon]}(\overline{B_{\epsilon/2}(c)})$ and thus $c'$ is an element of $p^{-1}p(\overline{B_{\epsilon/2}(c)})$.
\end{proof}
\begin{prop}\label{prop:4.13}
Let $\delta>0$ be given. There is an open $\Phi$-invariant $\VCyc$-cover $\mcU$ of $X_{\axes}$ whose dimension is at most $\dim(X)$ and for each $U\in\mcU$ there exists $x\in X$ with $U\subseteq B_\delta(\Phi_\bbR(x))$.
\end{prop}
To prove this we need the following lemmas.
\begin{lemma}
\label{lem:discrete}
For all $y\in Y_{\axes}$ the stabilizer $G_y\coloneqq \{g\in G\mid gy=y\}$ is virtually cyclic of type I and $Gy\subseteq Y_{\axes}$ is closed and discrete.
\end{lemma}
\begin{proof}
The space $y= \Phi_\bbR(c)\cong \bbR$ is a closed $G_y$-invariant subspace of $X$. Thus the group action of $G_y$ on $y$ is proper.  Furthermore we have a homomorphism $G_y\to \mathbbm{R},\; g\mapsto l(g,y)$, where we set $l(g,y)\coloneqq 0$ if $gy=y$. 
Since $c$ is not fixed under the flow, we can find a small $\varepsilon$ such that $\Phi_t(c)\notin Gc$ for all $t\in (0,\varepsilon)$ and thus the image of this homomorphism is discrete. It is nontrivial since $c$ is $G$-periodic. Thus it must be infinite cyclic.
Since the $G$-action on $X$ is proper the kernel of this map is finite. Hence $G_y$ is virtually cyclic of type $I$.

Next suppose that we have a sequence $g_i\Phi_{t_i}(c)\in p^{-1}(Gy)$ with $g_i\in G$ that converges to some $c'\in X_{\axes}$. 
Pick $g\in G$ such that $c$ is an axis for $g$ with $l(g,c)\leq \gamma$. Then by replacing $g_i$ by $g_ig^{m_i}$ for some $m_i\in\bbZ$ we can assume $t_i\in[0,\gamma]$. We have $g_ic\in \Phi_{[-\gamma,0]}B_1(c')=:L$ for $n$ big enough. Since the action is proper and $L$ is compact we can find a subsequence with $g_i\equiv h$. Furthermore we can pick a subsequence such that $\lim_{i\in \bbN}t_i$ exists. Thus $h\Phi_{t_i}c$ converges to $c'$ and $c'=\Phi_{\lim_{i\in\bbN}t_i}hc\in p^{-1}(Gy)$. 
Therefore, $Gy$ is closed. And since we can always find a subsequence with $g_i\equiv h$ every converging sequence in $Gy$ already contains its limit point infinitely often. This implies that $Gy$ is discrete.
\end{proof}

\begin{lemma}
\label{lem:Gmetrizable}
The space $G\backslash Y_{\axes}$ is locally compact and metrizable.
\end{lemma}
\begin{proof}
The space $X_{\axes}$ is second-countable by \cref{lem:4.10}. The quotient map $p\colon X_{\axes}\to G\backslash Y_{\axes}$ is open, because it is the quotient by a group action. Let $y$ be a point in $ G\backslash Y_{\axes}$ and $c\in X_{\axes}$ be a preimage. Let $U$ be an open neighborhood of $y$ and let $L\subseteq p^{-1}(U)$ be a compact neighborhood of $c$. Since $p$ is continuous and open, $p(L)$ is a compact neighborhood of $y$. Thus $ G\backslash Y_{\axes}$ is locally compact.

Points in $G\backslash Y_{\axes}$ are closed, since $Gy\subseteq Y_{\axes}$ is closed for every $y\in Y_{\axes}$ by \cref{lem:discrete}. For a closed subset $A\subseteq G\backslash Y_{\axes}$ and $c\in X_{\axes}$ with $p(c)\notin A$ there is an $\epsilon>0$ such that $\overline{B_\epsilon(c)}$ is compact and $B_\epsilon(c)\cap p^{-1}(A)=\emptyset$ and thus also $p(B_\epsilon(c))\cap A=\emptyset$. 
\begin{claim}
The set $B\coloneqq p(\overline{B_{\epsilon/2}(c)})$ is closed.
\end{claim}
Then the complement of $B$ is an open neighborhood of $A$ and it is disjoint from the open neighborhood $p(B_{\epsilon/2}(c))$ of $p(c)$. Hence, $G\backslash Y_{\axes}$ is regular. $G\backslash Y_{\axes}$ is second-countable since it is a quotient of a subspace of a second-countable space.
By Urysohn's metrization theorem \cite[Theorem~34.1]{munkres} the quotient $G\backslash Y_{\axes}$ is metrizable.

It remains to prove the claim: 
Let $c_i$ be a sequence in $\overline{B_{\epsilon/2}(c)}$, $t_i\in\bbR,g_i\in G$ such that the sequence $g_i\Phi_{t_i}c_i$ is converging to $c'\in X_{\axes}$. Let $h_i\in G$ be given such that $c_i$ is an axis for $h_i$, then there are $m_i\in\bbZ, t_i'\in[0,\gamma]$ such that $g_i\Phi_{t_i}c_i=g_ih_i^{m_i}\Phi_{t'_i}c_i$, therefore we can assume $t_i\in[0,\gamma]$. Since $\overline{B_{\epsilon/2}(c)}$ and $[0,\gamma]$ are compact we can assume that $c_i$ converges to $k$ and $t_i$ converges to $t$.
By \cref{lem:4.6} there is $g\in G$ such that $c'=g\phi_tk\in p^{-1}(p(\overline{B_{\epsilon/2}(c)}))$. This proves the claim.
\end{proof}
The big difference to \cite{flow} is that there assumptions on the geometry are used to define a metric on $G\backslash Y_{\axes}$. Here we just use metrization theorems and thus get rid of those assumptions. 
\begin{lemma}
\label{lem:4.12}
We have $dim(G\backslash Y_{\axes})\leq \dim(X)$.
\end{lemma}
\begin{proof}
$X_{\axes}$ is a metric space and hence completely regular. For every $x\in X_{\axes}$ there is a compact neighborhood $L$ of $x$. The space $H_x\coloneqq \{t\in\bbR\mid \Phi_t L\cap L\neq \emptyset\}$ is compact by \cref{lem:proper}. This implies by \cite[Theorem 2.3.2]{palaisslices} that there is a slice at $x$, i.e. there exists $S_x\subseteq X_{\axes}$ containing $x$ such that $\Phi_\bbR S_x\subseteq X_{\axes}$ is open and an $\bbR$-equivariant map $f:\Phi_\bbR S_x\to \bbR$ such that $f^{-1}(0)=S_x$. For $y\in Y_{\axes}$ let $U\coloneqq p(\Phi_\bbR S_{x})$ for some $x\in X_{\axes}$ with $p(x)=y$. This is an open neighborhood of $y$. We can define a section $s:U\to X_{\axes}$ by $s(y')=S_x\cap p^{-1}(y')$ for all $y'\in U$.
The continuity of the section follows from the alternative definition
$s(y')=\Phi_{-f(x')}x'$ for some $x'$ with $p(x')=y'$.
Since $Y_{\axes}$ is locally compact by \cref{lem:metrizable} there is for each $y$ a compact neighborhood $K_y$ of $y$ and a section $s:K_y\to X_{\axes}$.

By \cref{thm:subspace} and \cref{thm:compare} this implies that $\dim(K)\leq \dim(X_{\axes})\le \dim (X)$ and therefore $\text{locdim}(Y_{\axes})\leq \dim (X)$. Since $Y_{\axes}$ is metrizable it is paracompact and normal. This implies $\dim (Y_{\axes})= \text{locdim}(Y_{\axes})\leq \dim(X)$ by \cite[Proposition 3.4]{pears1975dimension}.

For every $y\in Y_{\axes}$ there is a compact neighborhood $L$ of $y$. Since the $G$-action on $Y_{\axes}$ has closed, discrete orbits by \cref{lem:discrete} the map $p\colon L\to G\backslash GL$ is finite-to-one. Both $L$ and $G\backslash GL$ are metrizable by \cref{lem:metrizable} and \cref{lem:Gmetrizable} and thus paracompact and normal. This implies
\[\dim(L)=\dim(G\backslash GL)\]
by \cite[Proposition~9.2.16]{pears1975dimension}.
As above we get
\[\dim(G\backslash Y_{\axes})= \text{locdim}(G\backslash Y_{\axes})=\dim(Y_{\axes})\leq \dim(X).\qedhere\]
\end{proof}

\begin{lemma}\label{lem:42}
Let $c\in X_{\leq\gamma}$ and let $\mcF$ be a family of subgroups. For any open $\mcF$-neighborhood $U\subseteq X_{\leq \gamma}$ of $\Phi_\bbR c$ there exists an open $\mcF$-neighborhood $V\subseteq U$ of $\Phi_\bbR c$ which is invariant under the flow.
\end{lemma}
\begin{proof}
Let $C$ be the complement of $GU$ and let $V$ be the intersection of $U$ with the complement of $\Phi_\bbR C$. Then $V$ contains $\Phi_\bbR c$ and is an $\mcF$-subset. It remains to show that $V$ is open or equivalently that $\Phi_\bbR C$ is closed. Since every element has $G$-period at most $\gamma$ and $C$ is $G$-invariant we have that $\Phi_\bbR C=\Phi_{[0,\gamma]}C$. This is closed since $[0,\gamma]$ is compact and $C$ is closed.
\end{proof}

\begin{proof}[Proof of \cref{prop:4.13}]
Let $y\in Y_{\axes}$ be given. By \cref{lem:discrete} the set $Gy\setminus\{y\}$ is closed and therefore $p^{-1}(y)$ and $p^{-1}(Gy\backslash y)$ are closed. Let $c\in p^{-1}(y)$ and $g\in G$ be such that $c$ is an axis for $g$. There is $\delta>\epsilon>0$ such that we have $B_\epsilon(\Phi_{[0,l(g,c)]}c)\cap p^{-1}(Gy\backslash y)=\emptyset$. Since $Gy\backslash y$ is invariant under $g$ and $g^n(B_\epsilon(\Phi_{[0,l(g,c)]}c))=B_\epsilon(\Phi_{[nl(g,c),(n+1)l(g,c)]}c)$ we conclude that also $B_\epsilon(p^{-1}(y))\cap p^{-1}(Gy\backslash y)=\emptyset$. It follows that $B_{\epsilon/2}(p^{-1}(y))$ is an open $\VCyc$-neighborhood of $p^{-1}(y)$ and thus by \cref{lem:42} contains an open $\VCyc$-neighborhood $V'_y$ which is invariant under the flow.

Then $V_y\coloneqq p(V'_y)$ is a $\VCyc$-neighborhood of $y$. Because $\pi\colon Y_{\axes}\to G\backslash Y_{\axes}$ is open, $\{\pi(V_y)\mid y\in Y_{\axes}\}$ is an open cover of $G\backslash Y_{\axes}$. By \cref{lem:4.12} there is a refinement $\mcW$ of dimension less or equal to $\dim(X)$. For any $W\in\mcW$ pick $y_W\in Y_{\axes}$ such that $W\subseteq \pi(V_{y_W})$. Now define
\[\mcV\coloneqq \{\pi^{-1}(W)\cap gV_{y_W}\mid W\in\mcW,g\in G\}.\]
This is an open $\VCyc$-cover because each $V_y$ is an open $\VCyc$-set. Its dimension is bounded by $\dim(X)$ because the dimension of $\mcW$ is bounded by $\dim(X)$ and for all $g\in G,y\in Y_{\axes}$ we have either $V_y=gV_y$ or $V_y\cap gV_y=\emptyset$. It is $G$-invariant because each $\pi^{-1}(W)$ is $G$-invariant. By construction $p^{-1}(\pi^{-1}(W)\cap gV_{y_W})\subseteq V'_{y_W}\subseteq B_\delta(p^{-1}(y_W))$. Pulling the cover $\mcV$ back along the quotient by the flow $X_{\axes}\rightarrow Y_{\axes}$ yields the desired $\VCyc$-cover.
\end{proof}

\section{Compact flow lines with bounded G-period}\label{sec:periodic}

In this section we want to cover the periodic part $X_{\cyc}$ with small $G$-period. In \cite{bartels2008equivariant} and \cite{flow} this subspace was just the fixed points of the flow, but here we do not want to assume this. In this section $p\colon X_{\cyc}\to Y_{\cyc}$ denotes the quotient map.

\begin{lemma}\label{lem:xcycclosed}
The subspace $X_{\cyc}\subseteq X$ is closed.
\end{lemma}
\begin{proof}
Let $x_n\in X_{\cyc}$ be a sequence converging to $x\in X$. We want to show that $x\in X_{\cyc}=X_{\cyc}'\cap X_{\le \gamma}$.
There are $g_n\in G,s_n\in[\gamma/2,\gamma]$ with $g_nx_n=\Phi_{s_n}x_n$. If we have $g_n,s_n$ with $s_n\in[\gamma/4,\gamma/2]$ we take $g_n^2,2s_n$ and so on. After passing to a subsequence we can assume $s_n$ converges to $s\in[\gamma/2,\gamma]$. Then the sequence $g_n\Phi_{-s_n}x_n=x_n$ converges to $x$ and thus also $g_n^{-1}\Phi_{s}x$ converges to $x$. Since the $G$ action is proper, we can find a subsequence with $g_n\equiv g$. Therefore, $\Phi_{s}x=gx$ and $x\in X_\le \gamma$. The subgroup generated by $g$ acts properly on the compact space $\Phi_\bbR(x_n)$ and thus $g$ has finite order. This implies $\Phi_{ms}x=x$, where $m$ is the order of $g$. Hence $x\in X'_{\cyc}$.
\end{proof}

In general $X_{\cyc}'\subseteq X$ need not be closed, as the example of the geodesic flow on the unit tangent bundle of the two dimensional flat torus shows. Thus the bound on the $G$-period is really crucial.

\begin{lemma} The spaces $Y_{\cyc}$ and $G\backslash Y_{\cyc}$ are locally compact and the induced $G$-action on $Y_{\cyc}$ is proper.
\end{lemma}
\begin{proof}
Since $X_{\cyc}$ is closed in $X$ and $X$ is locally compact, also $X_{\cyc}$ is locally compact. Let $y\in Y_{\cyc}$ respectively $y\in G\backslash Y_{\cyc}$ and let $x\in X_{\cyc}$ be a preimage of $y$. Let $U$ be an open neighborhood of $y$. Then $x$ has a compact neighborhood $L$ contained in the preimage of $U$. The image of $L$ in $Y_{\cyc}$ respectively $G\backslash Y_{\cyc}$ is a compact neighborhood of $y$, since quotient maps by group actions are open.

Now let $y\in Y_{\cyc}$ be given. Since $X_{\cyc}$ is locally compact this implies that $p^{-1}(y)$ has a compact neighborhood $L_y$.
To show that $G$ acts properly on $Y_{\cyc}$, it suffices to show that  the set
\[T\coloneqq \{g\in G\mid gp(L_y)\cap p(L_y)\neq \emptyset\}=\{g\in G \mid \exists x\in L_y,t\in \bbR :gx\in \Phi_tL_y\}\]
is finite for any point $y\in Y_{\cyc}$. Let $S\coloneqq \{g\mid \exists c\in L_y,t\in[0,\gamma]:\Phi_tc=gc\}$. This set is a subset of $S':=\{g\in G\mid g\Phi_{[0,\gamma]}L_y\cap \Phi_{[0,\gamma]}L_y \neq \emptyset\}$ and the latter is finite, since $G$ acts properly on $X$ and $\Phi_{[0,\gamma]}L_y$ is compact. So $S$ is finite. Any $g\in S$ has finite order since $\langle g\rangle$ acts properly on the compact space $\Phi_\bbR c$ for $c\in L_y\subset X_{\cyc}$ with $\Phi_tc=gc$. Thus the set $S^*=\{g^m\mid g\in S, m\in \bbZ\}$
is also finite. Now let us show that $T\subseteq S'\cdot S^*$. For $g\in T$ choose $x\in L_y$, $t\in \bbR$ such that $gx\in \Phi_tL_y$. Since $x$ has $G$-period $\le \gamma$, we can find an $h\in G$, such that $hx = \Phi_{t'}x$ for some $t'\in (0,\gamma]$. Usually $t'$ is the $G$-period of $x$; if that period is $0$, we can choose $t'$ as any positive number. By definition of $S$ we have $h\in S$. We get 
\[gh^mx = g\Phi_{mt'}x=\Phi_{mt'}gx\in \Phi_{mt'+t}L_y\]
and if we choose $m\in \bbZ$ suitably, we have $mt'+t\in [0,\gamma]$ and thus $gh^mx\in \Phi_{[0,\gamma]}L_y$ and thus $gh^m\in S'$. Hence
\[g=(gh^m)\cdot h^{-m}\in S'\cdot S^*\qedhere\]

\end{proof}

\begin{lemma} The space $Y_{\cyc,>0}$ and its quotient by the $G$-action are Hausdorff.
\end{lemma}
\begin{proof}
Pick two different points $y,y'\in Y_{\cyc,>0}$ and let $x,x'\in X_{\cyc,>0}$ be two preimages. Thus $\Phi_\bbR(x)$ and $\Phi_\bbR(x')$ are compact. There exist open disjoint $\Fin$-neighborhoods $U,U'$ of $\Phi_\bbR(x)$ and $\Phi_\bbR(x')$. By \cref{lem:42} they contain open $\Fin$-neighborhoods $V,V'$ which are invariant under the flow. Their images under $p$ then are the desired disjoint open sets.

We still have to show that the quotient by the $G$-action is Hausdorff. Let points $Gy\neq Gy'\in G\backslash Y_{\cyc,>0}$ be given. Let $x,x',U,U'$ be as above. Since the $G$-action is proper, we can assume $U\cap gU'=\emptyset$ for all $g\in G$. Doing the same construction as above and pushing it to the quotient, we end up with separating neighborhoods for $Gy$ and $Gy'$. Thus $G\backslash Y_{\cyc,>0}$ is Hausdorff.
\end{proof}

\begin{lemma} The spaces $Y_{\cyc,>0}$  and its quotient by the $G$-action are paracompact and normal.
\end{lemma}
\begin{proof}
Since $Y_{\cyc,>0}$ and $G\backslash Y_{\cyc,>0}$ are open subsets of $Y_{\cyc}$ and $G\backslash Y_{\cyc}$ respectively, both are again locally compact.
By \cite[Exercise~3~on~p.205]{munkres} every locally compact, Hausdorff space is regular. Second-countable spaces are Lindel\"of spaces. A regular Lindel\"of space is paracompact by \cite[Theorem~41.5]{munkres}. All paracompact Hausdorff spaces are normal by \cite[Theorem~41.1]{munkres}.
\end{proof}

\begin{lemma}\label{lem:perdimBound} We have that
$\dim(G\backslash Y_{\cyc,>0})\le \dim(Y_{\cyc,>0})\le \dim(X)$.
\end{lemma}
\begin{proof}
We start with the second inequality. Let $y\in Y_{\cyc,>0}$ be any point and let $x$ be a preimage. Pick a box $B$ around $x$ and consider the map induced by $p$
\[p':S_B\rightarrow p(S_B).\]
Note that $p(U) = p(\Phi_{(-\varepsilon,\varepsilon)}U)$ and $p$ is open since it is the quotient by a group action. Thus the map $p'$  is a continuous open surjection.

Pick some $y'\in p'(S_B)$ and some preimage $x'$. Since $B$ is a box of length $l_B$, any two points of the set
\[S=\{t\in \bbR\mid \Phi_t(x')\in S_B\}\]
have distance at least $l_B/2$ and thus it is discrete. Let $M$ be the period of $x'$.
Thus,
\[\{x''\in S_B\mid p(x'')=y'\}=\{\Phi_s(x')\mid s\in S\}=\{\Phi_s(x')\mid s\in S\cap [0,M]\}\]
is finite. Thus we have a continuous surjection between paracompact, normal spaces where every point has finitely many preimages.
Thus by \cite[Proposition~9.2.16]{pears1975dimension} we get that $\dim(p(S_B))\le \dim(S_B)\le  \dim(X)$.
Since $p'(\mathring{S}_B)=p'(\mathring{B})$ is open and $x$ was arbitrary, we get that $\text{locdim}(Y_{\cyc,>0})\leq \dim(X)$. 

The space $Y_{\cyc,>0}$ is paracompact and normal and by \cite[Proposition~3.4]{pears1975dimension} this implies $\dim (Y_{\cyc,>0})= \text{locdim} (Y_{\cyc,>0})\leq \dim (X)$.

The first inequality follows the same way. Note that every point in $Y_{\cyc,>0}$ has a compact, $\Fin$-neighborhood $K$. To understand the local dimension, we can consider the map 
$K\rightarrow G_K\backslash K\subseteq G\backslash Y_{\cyc,>0}$.
\end{proof}

\begin{lemma}
\label{lem:4.15renew} Let $\delta>0$ be given. There is a $G$-invariant $\Fin$-cover $\mcV_{\cyc}$ of $X_{\cyc}$ of dimension at most $2\dim(X)+1$ such that we can find for every $x\in X_{\cyc}$ a $V\in \mcV_{\cyc}$ containing $\Phi_\bbR(x)$ and for every $V\in\mcV_{\cyc}$ a point $x\in X$ with $V\subseteq B_\delta(\Phi_\bbR(x))$.
\end{lemma}
\begin{proof}
We will first deal with the subspace $X^\bbR$ which is independent of $\gamma$.  Let $\pi:X^\bbR\rightarrow G\backslash X^\bbR$ denote the quotient map. Since the group action is proper, we get that the quotient is metrizable. 
For any point $Gx\in G\backslash X$ pick a compact neighborhood $K$ of a preimage. The quotient map
$K\rightarrow G\backslash GK$ is a continuous, finite-to-one, open surjection. 
Thus by \cite[Proposition~9.2.16]{pears1975dimension} we get that $\dim(G\backslash GK)\le \dim(K)\le  \dim(X)$.
Since $x$ was arbitrary, we get that $\text{locdim}(G\backslash X^\bbR)\leq \dim(X)$. 

The space $G\backslash X^\bbR$ is paracompact and normal and by \cite[Proposition~3.4]{pears1975dimension} this implies $\dim (G\backslash X^\bbR)= \text{locdim} (G\backslash X^\bbR)\leq \dim (X)$.
Pick a $\Fin$-cover $\calv$ of $X^\bbR$ and refine $\{\pi(V\cap B_{\delta/2}(x))\mid V \in \calv, x\in X^{\bbR}\}$ to a cover $\calu$ of dimension at most $\dim(X)$.
Since $\calu$ is a refinement of $\{\pi(V)\mid V\in \calv \}$, we can find for every $U\in \calu$ a $V_U\in \calv$ with $U\subseteq \pi(V_U)$. 
Then the open $\Fin$-cover of $X^\bbR$ given by 
\[\calu_{X^\bbR}=\{\pi^{-1}(U)\cap  gV_{U} \mid U\in \calu,g\in G\}\]
is at most $\dim(X)$ dimensional.

Now let us look at $X_{\cyc,>0}$. Let $p\colon X_{\cyc,>0}\to Y_{\cyc,>0}$ be the quotient map.
Since the $G$-action on $Y_{\cyc,>0}$ is proper, we can find a refinement $\calv$ of $\{p(B_{\delta/2}(\Phi_\bbR(x)))\mid x\in X_{\cyc,>0}\}$ that is a $\Fin$-cover of $Y_{\cyc,>0}$.

Let us look at the quotient by the group action
\[\pi:Y_{\cyc,>0}\rightarrow G\backslash Y_{\cyc,>0}.\]
Push $\calv$ down to the quotient by the $G$-action and refine it to a cover $\calu$ of dimension at most $\dim(X)$ using \cref{lem:perdimBound}. Being a refinement means that we can find for every $U\in \calu$ a $V_U\in \calv$ with $\pi^{-1}(U)\subseteq G(V_U)$. Then the open $\Fin$-cover
\[\calu_{Y_{\cyc,>0}}=\{\pi^{-1}(V)\cap  gU_V \mid V\in \calv,g\in G\}\]
is a cover of dimension at most $\dim(X)$ of $Y_{\cyc,>0}$. Pulling it back to $X_{\cyc,>0}$ gives an open $\Fin$-cover $\calu_{X_{\cyc,>0}}$ of $X_{\cyc,>0}$. 
Now use \cref{lem:openextension} to extend the collections $\calu_{X^\bbR},\calu_{X_{\cyc,>0}}$ to open subsets of $B_{\delta/2}(X^\bbR)\cap X_{\cyc}$ and $B_{\delta/2}(X_{\cyc,>0})\cap X_{\cyc}$ respectively and take $\mcV_{\cyc}$ as their union. The construction in \cref{lem:openextension} is made in such a way that set $V\in\mcV_{\cyc}$ is still contained in $B_\delta(\Phi_\bbR(x))$ for some $x\in X_{\cyc}$.
\end{proof}
\section{Proof of the Main Theorem}\label{sec:proofmain}
We can now use the previous sections to prove the main theorem.

\begin{proof}[Proof of the Main \cref{thm:main}]
Let $\gamma\coloneqq 20\alpha$. Let $\calu_{>\gamma}$ be the open cover of $X_{>\gamma}$ from \cref{thm:long}. Its dimension is at most $5(\ind(X)+1)=5\dim(X)+5$. Recall that by \cref{lem:4.10} and \cref{lem:xcycclosed} we have $X_{\le \gamma}=X_{\axes}\amalg X_{\cyc}$ and $X_{\le \gamma}$ is closed in $X$. Thus the elements of $\calu_{>\gamma}$ are open in $X$.

To cover $X_{\le\gamma}$ we can use the covers from \cref{lem:4.15renew} and \cref{prop:4.13} with $\delta/2$ instead of $\delta$ and take their union. The union has dimension at most $2\dim(X)+1$. Extend it to a $\VCyc$-collection $\calu_{\le\gamma}$ of open subsets of $X$ using \cref{lem:openextension} for $X_{\le \gamma}\subseteq B_{\delta/2}(X_{\le\gamma})$ and define $\calu\coloneqq \calu_{\le \gamma}\cup \calu_{>\gamma}$. It has dimension at most $7\dim(X)+7$. Since we are only extending to a $\delta/2$-neighborhood the construction in \cref{lem:openextension} will enlarge sets by no more than $\delta/2$ and each $U\in\calu$ is still contained in $B_\delta(\Phi_\bbR(x))$ for some $x\in X$.
\end{proof}

\begin{rem}
Only the construction for the nonperiodic part with short $G$-period produces a $\VCyc$-cover. For the other parts the construction produces a $\mcF in$-cover instead. Therefore, only those virtually cyclic subgroups of $G$ appear as stabilizers of the cover for which there exists an axis. 
\end{rem}

If the flow space is cocompact, then it also is locally compact and second-countable by \cref{prop:sep} and as in \cite[Lemma~5.8]{flow} we obtain the following corollary.

\begin{cor}\label{cor:main}
Let $X$  be a finite-dimensional and cocompact flow space $X$ and $\alpha$ be a number greater than $0$. Then there is an $\varepsilon>0$ and a $\VCyc$-cover $\calu$ of $X$ such that for every point $x\in X$ there is an open set $U\in \calu$ with $B_{\epsilon}(\Phi_{[-\alpha,\alpha]}(x))\subseteq U$ of dimension at most $7\dim(X)+7$.
\end{cor}

\section{Applications}
\label{sec:applications}
\subsection*{Approximating the orbit space}

Let $U$ be a simplicial complex and $v\in U$ a vertex. The open star $\text{St}(v)$ of $v$ consists of all simplices of $U$ containing $v$. Note that this is in general not a subcomplex. From \cref{thm:main} we obtain the following approximation of the orbit space $\bbR\backslash X$ already eluded to in the introduction.

\begin{theorem}
Let $X$ be as in \cref{thm:main}. Then there exists a sequence of $G$-equivariant maps $f_n\colon X\to V_n$ into simplicial complexes $V_n$ with dimension at most $7\dim X+7$ and $\VCyc$-stabilizers such that for every vertex $v\in V$ there is $x\in X$ with $f_n^{-1}(\text{St}(v))\subseteq B_\delta(\Phi_\bbR(x))$ and for every $x\in X$ there exists a vertex $v\in V_n$ with $\Phi_{[-n,n]}(x)\subseteq f_n^{-1}(\text{St}(v))$.
\end{theorem}
\begin{proof}
Let $\mcU_n$ be a cover as in \cref{thm:main} such that for each $x\in X$ there exists $U\in\mcU_n$ with $\Phi_{[-n,n]}(x)\subseteq U$. Let $V_n$ be the nerve of $\mcU_n$, i.e.\ the simplicial complex with vertex set $\mcU_n$ and  the elements $U_1,\ldots, U_n$ span a simplex if and only if $\bigcap_{i=1}^nU_i\neq \emptyset$. For $x\in X$ let $d(x):=\sum_{U\in \mcU_n}d(x,X\setminus U)$. We have $d(x)\in\bbR$, since the cover is locally finite. Define $f_n\colon X\to V_n$ by $x\mapsto \sum_{U\in\mcU_n}\frac{d(x,X\setminus U)}{d(x)}U$. Then $f_n^{-1}(\text{St(U)})=U$ and hence $V_n$ satisfies the properties in the theorem.
\end{proof}

\subsection*{The Farrell--Jones conjecture}
The Farrell--Jones conjecture for a group $G$  says that the $K$-theoretic assembly map
\[\mathcal{H}^G_*(E_\VCyc G;\mathbf{K}_\cala)\rightarrow \mathcal{H}^G_*(pt;\mathbf{K}_\cala)=K^{alg}_*(\cala[G])\]
and the $L$-theoretic assembly map
\[\mathcal{H}^G_*(E_\VCyc G;\mathbf{L}_\cala)\rightarrow \mathcal{H}^G_*(pt ;\mathbf{L}_\cala)=L^{\langle-\infty\rangle}_*(\cala[G])\]
are isomorphisms for any additive $G$-category $\cala$ (with involution), see \cite[Conjectures~3.2~and~5.1]{coefficients}. The Farrell--Jones conjecture implies several other conjectures. See \cite{baum} for details.

As in \cite[Definition~2.15]{wegner2013farrell} we say that a group $G$ satisfies the Farrell--Jones conjecture with finite wreath products if for any finite group $F$ the wreath product $G\wr F$ satisfies the $K$- and $L$-theoretic Farrell--Jones conjecture. We will use the abbreviation \FJCw{} for ”Farrell--Jones conjecture with finite wreath products”.

\begin{definition}[{\cite[Definition~2.1,Definition~2.3]{wegnercat}}] A \emph{strong homotopy action} of a group $G$ on a topological space $X$
is a continuous map
\[\Psi: \coprod_{j=0}^\infty (G\times [0,1])^j\times G\times X \rightarrow X\]
with the following properties:
\begin{enumerate}
\item $\Psi(\ldots,g_l,0,g_{l-1},\ldots) = \Psi(\ldots,g_l,\Psi(g_{l-1},\ldots))$,
\item $\Psi(\ldots,g_l,1,g_{l-1},\ldots) = \Psi(\ldots,g_l \cdot g_{l-1},\ldots)$,
\item $\Psi(e,t_j,g_{j-1},\ldots)=\Psi(g_{j-1},\ldots)$,
\item $\Psi(\ldots,t_l,e,t_{l-1},\ldots) = \Psi(\ldots,t_l \cdot t_{l-1},\ldots)$,
\item $\Psi(\ldots,t_1,e,x) = \Psi(\ldots,x)$,
\item $\Psi(e,x) = x$.
\end{enumerate}
For a subset $S\subseteq G$ containing $e$,$g\in G$ and a $k\in \bbN$ define
\[F_g(\Psi,S,k)\coloneqq  \{\Psi(g_k,t_k,\ldots,g_0,?):X\rightarrow X\mid g_i\in S,t_i\in [0,1],g_k \ldots g_0=g\}.\]
For $(g,x)\in G\times X$ we define $S^0_{\Psi,S,k}(g,x)$ as $\{(g,x)\}$, $S^1_{\Psi,S,k}(g,x)\subseteq G\times X$ as the subset consisting of all $(h,y)\in G\times X$ with the following property: There are $a,b\in S$, $f\in F_a(\Psi,S,k),f'\in F_b(\Psi,S,k)$ such that $f(x)=f'(y)$ and $h=ga^{-1}b$.

For $n\ge 2$ define inductively $S^n_{\Psi,S,k}(g,x)=\bigcup_{(h,y)\in S^{n-1}_{\Psi,S,k}(g,x)} S^1_{\Psi,S,k}(h,y)$.
\end{definition}
The definition of a controlled $N$-dominated metric space can be found in \cite[Definition~1.5]{bartels2012borel}.

\begin{definition}[{\cite[Definition~3.1]{wegnercat}}] A group $G$ is \emph{strongly transfer reducible} over a family $\calf$ of subgroups if there exists a natural number $N\in \bbN$ with the following property: For every finite symmetric subset $S\subseteq G$ containing the trivial element $e$ and all $n,k\in \bbN$ there are
\begin{itemize}
\item a compact, contractible, controlled $N$-dominated metric space $X$,
\item a strong homotopy $G$-action $\Psi$ on $X$ and
\item an open $\calf$-cover $\calu$ of $G\times X$ of dimension at most $N$ such that for every $(g,x)\in G\times X$ there exists $U\in \calu$ with $S^n_{\Psi,S,k}(g,x)\subseteq U$.
\end{itemize}
\end{definition}
Every virtually cyclic group is a CAT(0)-group and therefore satisfies \FJCw{} by \cite[Example~2.16(i)]{wegner2013farrell}. Thus, by \cite[Proposition~2.20]{wegner2013farrell} a group $G$ satisfies \FJCw{} if it is strongly transfer reducible over the family $\VCyc$ of virtually cyclic subgroups. We  will now define certain properties of flow spaces under which \cref{thm:main} allows us to show strong transfer reducibility.
\begin{defi}
\label{def:unicont}
Let $X$ be a flow space. The flow is  \emph{uniformly continuous} if for every $\alpha,\varepsilon > 0$ there is a $\delta>0$ such that for all $z,z'\in X$ with $d_X(z,z')\leq \delta$ and for any $t\in[-\alpha,\alpha]$ we also get $d_X(\Phi_tz,\Phi_tz')\leq \varepsilon$.
\end{defi}

The following definition is a weakening of \cite[Definition~5.5]{flow}. We do not have to assume the existence of covers of the periodic part anymore. 

\begin{definition}\label{def:longinf} A flow space $X$ for a group $G$ \emph{admits long $\calf$-covers at infinity} if the following holds:
There is $M>0$ such that for every $\alpha>0$ there is an $\calf$-collection $\calv$ of dimension at most $M$, a compact subset $K\subseteq X$ and an $\varepsilon>0$ such that for every $z\in X\setminus GK$ there is a $V\in \calv$ with
\[B_\varepsilon(\Phi_{[-\alpha,\alpha]}x)\subseteq V.\]
\end{definition}
Note that in the definition it makes no difference if we assume that collection $\calv$ consists of open $\calf$-sets.
\begin{lemma}\label{lem:5.7} Let $X$ be a finite-dimensional, second-countable and locally compact flow space for the group $G$ such that
 there are long $\calf$-covers at infinity. Then there is an $N>0$ such that for every $\alpha>0$ there is an $\varepsilon>0$ and an $\calf\cup \VCyc$-cover $\calu$ of $X$ of dimension at most $N$ such that for every $x\in X$ there is a $U\in \calu$ with
 \[B_\varepsilon(\Phi_{[-\alpha,\alpha]}x)\subseteq U.\]
\end{lemma}
\begin{proof} Let $\alpha>0$ be given. Since $X$ admits long covers at infinity, we obtain an $M>0$ (which is independent of $\alpha$), an $\calf$-collection $\calv$ of dimension at most $M$, an $\varepsilon_\infty>0$ and a compact subset $K$ as in \cref{def:longinf}.
We can find a $\VCyc$-cover $\calv'$ of $X$ as in \cref{thm:main}. Using \cite[Lemma~5.8]{flow} we can find an $\varepsilon_K>0$ such that we have for every $x\in GK$ an open set $U\in \calv'$ with $B_{\varepsilon_K}(\Phi_{[-\alpha,\alpha]}(x))\subseteq U$. Now the $\calf\cup \VCyc$-cover $\calu\coloneqq \calv'\cup \calv$ has dimension at most $N\coloneqq M+7\dim(X)+8$ and we can find for $x\in X$ an open set $U\in \calu$ with:
\[B_{\varepsilon}(\Phi_{[-\alpha,\alpha]}(x))\subseteq U\]
with $\varepsilon\coloneqq \min(\varepsilon_\infty,\varepsilon_K)$. 
\end{proof}

Morally, long covers at infinity allow us to find an $\varepsilon$ as in the last lemma even in the noncocompact setting.

\begin{definition} \label{def:sct}A flow space FS for a group $G$ admits \emph{strong contracting transfers} if there is an $N\in\bbN$ such that for every finite subset $S$ of $G$ and every $k\in \bbN$ there exists $\beta>0$  such that the following holds. For every $\delta>0$ there is 
\begin{enumerate}
\item $T>0$;
\item a contractible, compact, controlled $N$-dominated space $X$;
\item a strong homotopy action $\Psi$ on $X$;
\item a $G$-equivariant map $\iota: G\times X\rightarrow FS$ (where the $G$-action on $G\times X$ is given by $g\cdot (g',x)=(gg',x)$) such that the following holds:
\item \label{def:sct5} for every $(g,x)\in G\times X, s\in S,f\in F_g(\Psi,S,k)$ there is a $\tau \in [-\beta,\beta]$ such that
$d_{FS}(\Phi_T\iota(g,x),\Phi_{T+\tau}\iota(gs^{-1},f(x)))\le \delta$.
\end{enumerate}
\end{definition}

\begin{lemma}\label{lem:5.12} Let $FS$ be a flow space for a group $G$ with a uniformly continuous flow and assume that $FS$ admits strong contracting transfers. Let $\varepsilon >0, k,n\in \bbN$ and a finite subset $S\subseteq G$ containing $e$ be given. Let $\beta$ be as in \cref{def:sct} and define $\alpha\coloneqq 2n\beta$. Let $\delta$ be as in \cref{def:unicont}. Let $T>0, X, \Psi$ and $\iota$ be as in \cref{def:sct}. Then for every $(g,x)\in G\times X$ and $(h,y)\in S^n_{\Psi,S,k}(g,x)$ there is a $\tau\in [-\alpha,\alpha]$ such that 
\begin{equation*}d_{FS}(\Phi_T(\iota(g,x)),\Phi_{T+\tau}(\iota(h,y)))\le 2n\varepsilon.\end{equation*}
\end{lemma}
\begin{proof}
We will prove by induction on $m=0,\ldots, n$ that for every $(h,y)\in S^m_{\Psi,S,k}(g,x)$ we can find a $\tau\in [-2m\beta,2m\beta]$ such that
\begin{equation*}d_{FS}(\Phi_T(\iota(g,x)),\Phi_{T+\tau}(\iota(h,y)))\le 2m\varepsilon.\end{equation*}
This is clear for $m=0$. For $(h,y)\in S^{m+1}_{\Psi,S,k}(g,x)$ choose a $(g',x')\in S^m_{\Psi,S,k}(g,x)$ with $(h,y)\in S^1_{\Psi,S,k}(g',x')$. Thus, there are $a,b\in S,f\in F_a(\Psi,S,k),f'\in F_b(\Psi,S,k)$ such that $f(x')=f'(y)$ and $hb^{-1}=g'a^{-1}$.
By induction assumption and \cref{def:sct}~\eqref{def:sct5} there are $\tau\in[-2m\beta,2m\beta],\tau_f,\tau_{f'}\in[-\beta,\beta]$ such that
\begin{eqnarray*}
d_{FS}(\Phi_T(\iota(g,x)),\Phi_{T+\tau}(\iota(g',x')))&\le& 2m\varepsilon,\\
d_{FS}(\Phi_T\iota(g',x'),\Phi_{T+\tau_f}\iota(ga^{-1},f(x')))&\le& \delta,\\
d_{FS}(\Phi_T\iota(h,y),\Phi_{T+\tau_{f'}}\iota(hb^{-1},f'(y)))&\le& \delta.
\end{eqnarray*}
By uniform continuity we get
\begin{eqnarray*}
d_{FS}(\Phi_{T+\tau}\iota(g',x'),\Phi_{T+\tau+\tau_f}\iota(ga^{-1},f(x')))&\le& \varepsilon,\\
d_{FS}(\Phi_{T+\tau+\tau_f-\tau_{f'}}\iota(h,y),\Phi_{T+\tau+\tau_{f}}\iota(hb^{-1},f'(y)))&\le& \varepsilon.
\end{eqnarray*}
Let $\tau'\coloneqq \tau+\tau_f+\tau_{f'}\in[-2(m+1)\beta,2(m+1)\beta]$ and by the triangle inequality we obtain
\[d_{FS}(\Phi_T(\iota(g,x)),\Phi_{T+\tau'}(\iota(h,y)))\le 2m\varepsilon+\varepsilon+\varepsilon=2(m+1)\varepsilon.\qedhere\]
\end{proof}
\cref{lem:5.7} is a generalization of \cite[Theorem~5.7]{flow}. Using this and \cref{lem:5.12} instead of \cite[Lemma~5.12]{flow} as in the proof of \cite[Proposition~5.11]{flow} we obtain the following proposition.

\begin{proposition} Let $X$ be a finite-dimensional, second-countable and locally compact flow space for the group $G$ such that
\begin{enumerate}
\item the flow is uniformly continuous,
\item $X$ admits strong contracting transfers and
\item there are long $\calf$-covers at infinity.
\end{enumerate}
Then $G$ is strongly transfer reducible with respect to the family $\VCyc\cup \calf$.
\end{proposition}
\begin{cor} If $X$ is a cocompact, finite-dimensional flow space for the group $G$, which admits strong contracting transfers, then $G$ is strongly transfer reducible with respect to the family $\VCyc$, in particular $G$ satisfies \FJCw{}.
\end{cor}
\begin{proof} By \cref{prop:sep}, $X$ is locally compact and second-countable. Cocompactness implies uniform continuity and 
choosing $\calv = \emptyset$ and $K$ as a compact subset of $X$ with $GK=X$ shows that $X$ admits long covers at infinity.
Then $G$ satisfies the Farrell--Jones conjecture by \cite[Proposition~2.20]{wegner2013farrell}.
\end{proof}

\bibliographystyle{amsalpha}
\bibliography{covers}
\end{document}